\documentclass[10pt, a4paper, reqno]{amsart}

\usepackage{amscd}
\usepackage{dsfont}
\usepackage{tikz, xcolor}

\usetikzlibrary{calc}
\usepackage{array}
\usetikzlibrary{arrows}\usepackage{tkz-graph}
\usetikzlibrary{decorations.pathreplacing}

\usepackage{pgfplots}

\usepackage{amsmath}
\usepackage{amstext}
\usepackage{amsthm}
\usepackage{amssymb}
\usepackage{amsopn}
\usepackage{amssymb}
\usepackage{amsxtra}
\usepackage{amsfonts}
\usepackage{fancyhdr}
\usepackage{paralist, epsfig}
\usepackage[mathcal]{euscript}
\usepackage[english]{babel}
\usepackage[latin1]{inputenc}

\makeatletter
\g@addto@macro\normalsize{%
  \setlength\abovedisplayskip{10pt}
  \setlength\belowdisplayskip{10pt}
  \setlength\abovedisplayshortskip{5pt}
  \setlength\belowdisplayshortskip{8pt}
}

\allowdisplaybreaks

\newtheoremstyle{normal}
{5pt}
{5pt}
{\normalfont}
{}
{\bfseries}
{}
{0.4em}
{\bfseries{\thmnumber{#2}.\,\thmname{#1}.\thmnote{ \hspace{0.5em}(#3)\newline}}}

\theoremstyle{normal}

\newtheorem{thm}{Theorem}
\newtheorem{ex}[thm]{Example}
\newtheorem{ta}[thm]{Exercise}
\newtheorem{app}[thm]{Application}
\newtheorem{dfn}[thm]{Definition}

\renewcommand{\epsilon}{\varepsilon}

\usepackage{setspace}
\usepackage{color}

\definecolor{grey}{gray}{.3}

\definecolor{green-1}{HTML}{0080FF} 
\definecolor{grey-1}{HTML}{BDBDBD}
\definecolor{grey-2}{HTML}{A4A4A4}
\definecolor{grey-3}{HTML}{848484}%
\definecolor{grey-4}{HTML}{848484}%

\definecolor{ocker-1}{HTML}{0080FF}

\definecolor{background}{HTML}{E6E6E6}
\definecolor{red-1}{HTML}{0080FF}

\begin{document}
$ $
\vspace{-60pt}

\title{An Open Day in the Metric Space}

\author{Sven-Ake Wegner\hspace{0.1pt}\MakeLowercase{$^{\text{1}}$} and Katrin Rolka\hspace{0.5pt}\MakeLowercase{$^{\text{2}}$}}

\renewcommand{\thefootnote}{}
\hspace{-1000pt}\footnote{\hspace{5.5pt}2010 \emph{Mathematics Subject Classification}: Primary 97C90; Secondary 97D40, 97D80.\vspace{1.6pt}}

\hspace{-1000pt}\footnote{\hspace{5.5pt}\emph{Key words and phrases}: workshop for high school students, metric space, open days in mathematics, student life cycle.\vspace{1.6pt}}

\hspace{-1000pt}\footnote{\hspace{0pt}$^{1}$\,Nazarbayev University, Department of Mathematics, School of Science and Technology, 53 Kabanbay Batyr Ave,\linebreak\phantom{x}\hspace{1.2pt}010000 Astana, Kazakhstan, Phone: +7\hspace{1.2pt}(8)\hspace{1.2pt}7172\hspace{1.2pt}/\hspace{1.2pt}69\hspace{1.2pt}-\hspace{1.2pt}4671, e-mail: svenake.wegner@nu.edu.kz.\vspace{1.6pt}}

\hspace{-1000pt}\footnote{\hspace{0pt}$^{2}$\,Ruhr-Universit\"at Bochum, Fakult\"at f\"ur Mathematik, Universit\"atsstr.\,150, 44780 Bochum, Germany, Phone:\hspace{1.2pt}\hspace{1.2pt}+49\hspace{1.2pt}(0)\linebreak\phantom{x}\hspace{1pt}234\hspace{1.2pt}/\hspace{1.2pt}32\hspace{1.2pt}-\hspace{1.2pt}25616, e-mail: katrin.rolka@rub.de.\vspace{1.6pt}}

\begin{abstract} We report on a workshop for grade eleven high school students, which took place in the framework of a university open day. During the workshop the participants first discovered the key properties of the intuitive concept of distance from real life examples. After this preparation, the formal definition of a metric space was introduced and discussed in small groups by means of problem-oriented exercise sessions.
\end{abstract}

\maketitle

\vspace{-15pt}

\section{Introduction}\label{SEC-1}

In the last decade, the number of events organized by the universities for students at high school or even elementary school has increased significantly. This can be seen, e.g., by browsing corresponding university websites. Publications, however, are mostly informal and often concentrate on compiling the material used rather than on explaining and reflecting the methodological and didactic concepts involved. One of the few exceptions is a paper by Halverscheid and Sibbertsen \cite{HS} in which the authors conduct a careful analysis of the students' feedback. The aim of many events is to recruit future university students in the respective discipline. In view of the high number of university drop-outs---in particular in mathematics---it is worthwhile to design recruiting events in such a way which allows the participants to gain a realistic insight into their potential life as a student rather than merely to maximize the number of future students. In order to achieve this goal we suggest using a format which allows applicants to imitate the main activities that university students have to pursue during their mathematics studies, in particular those of the first year at university: attending lectures, solving exercises on their own or in small groups and presenting solutions to classmates. In addition, we suggest that the level of the lectures and exercises should ever so often reach university level. In order to meet the latter goal, one first has to select an appropriate topic. On the one hand, the topic needs to be accessible for high school students and on the other hand it should allow the teacher to illustrate what university studies are about. Second, a workshop which covers the above three activities (attending lectures, solving exercises, presenting solutions) has to be conceptualized.

\smallskip

In this article we report on a workshop for grade eleven high school students organized in the summer term 2013 at the University of Wuppertal (Germany) and whose topic was the mathematical notion of a metric space. In Section \ref{SEC-2} we outline the material covered and we describe our work flow during the workshop. In Section \ref{SEC-3} we present and discuss the feedback received from the participants.

\vspace{2pt}

\section{Workshop}\label{SEC-2}

Our workshop entitled \textquotedblleft{}Metric Spaces---on the notion of distance in mathematics\textquotedblright{} took place in the summer term 2013 at the University of Wuppertal (Germany). It was part of the Wuppertal Summer University, an open day for female high school students interested in mathematics, natural sciences or engineering. One specific aim was to complement other formats like lectures, presentations or panel discussions by an event that allows the participants to take an active role and enables them to experience what it feels like to study mathematics on a university level. The workshop took place in a seminar room and lasted for two hours. Eleven high school students, between 16 and 18 years old, participated. Due to the format of the open day, the participants came from different high schools with wildly varying backgrounds in mathematics. German high school education usually treats concepts like mappings, limits or continuity not in a rigorous but rather in an intuitive way. In order to allow the students to work on the exercises together, they were split into four groups. Since we had recruited four undergraduate students for the workshop we were able to assign one teaching assistant to each group to provide support during the exercise sessions.

\smallskip

In order to make our workshop on the basis of this article easily reproducible we give in the sequel an outline of the lecture, the exercises and their solutions. We note that the following is a selection of standard material as presented in courses on analysis and is covered together with more background information in many textbooks. We refer, e.g., to Rudin \cite[Chapter 2, p.30ff]{Rudin}, Browder \cite[Chapter 6.2]{Browder},  Amann and Escher \cite[Chapter II.1]{AE1} and Dugundji \cite[Chapter IX]{D}. For further reading about the Manhattan metric, see Exercise \ref{Exercise-1} and Theorem \ref{THM}, we refer to the book by Krause \cite{Krause}.

\medskip

We started the workshop with the following exercise for the students.

\medskip

\begin{ta}\label{Exercise-1} Find a meaningful notion of distance in the following four situations.\vspace{3pt}
\begin{compactitem}
\item[(1)] Points in the plane.\vspace{3pt}
\item[(2)] Places in the city.\vspace{3pt}
\item[(3)] Members of a social network.\vspace{3pt}
\item[(4)] Continuous functions $f$, $g\colon[0,1]\rightarrow\mathbb{R}$, e.g., $f(x)=\sqrt{x}$, $g(x)=x^3$.
\end{compactitem}
\end{ta}

\medskip

After an individual working period of about fifteen minutes we presented a possible solution and discussed the results found by the students during the prior working period.

\medskip

\begin{compactitem}

\item[(1)] Points $P_0$ and $P_1$ in the plane can be described by coordinates $P_0=(x_0,y_0)$, $P_1=(x_1,y_1)$. The \textquotedblleft{}usual distance\textquotedblright{}, i.e., the length of a straight line from $P_0$ to $P_1$, can then be expressed via the Pythagorean theorem in terms of the coordinates, that is
$$
d(P_0,P_1)=\sqrt{(y_1-y_0)^2+(x_1-x_0)^2},
$$
compare Figure 1.\vspace{-10pt}
\begin{figure}[h]
\hspace{17pt}\begin{tikzpicture}[scale=1.3]
    \draw [<->,thick] (0,1.9) node (yaxis) [above] {}
        |- (3.1,0) node (xaxis) [right] {};
    \coordinate[label=above:$P_0$] (A) at (0.5,0.55);
\coordinate[label=above:$P_1$] (B) at (2.5,1.6);
    \coordinate (C) at (2.5,0.55);

??\draw (A) -- (B);
\fill (A) circle (1pt);
\fill (B) circle (1pt);

\draw[black,line width=0.7pt] (A) -- (C) -- (B);

\draw[dashed] (yaxis |- A) node[left] {$y_0$} 
    -| (xaxis -| A) node[below] {$x_0$};
    
    \draw[dashed] (yaxis |- B) node[left] {$y_1$};
    \draw[dashed] (xaxis -| C) node[below] {$x_1$};
    \draw[dashed] (2.5,0) -- (C);
    \draw[dashed] (0,1.6) -- (B);

\draw [decorate,decoration={brace,amplitude=4pt, mirror,raise=6.2pt},yshift=-40pt,xshift=30pt]
(A) -- (C) node [black,midway,yshift=-0.5cm] {\footnotesize
$x_1-x_0$};

\draw [decorate,decoration={brace,amplitude=4pt,mirror,raise=6.2pt},yshift=0pt]
(C) -- (B) node [black,midway,xshift=0.8cm] {~~~\footnotesize$y_1-y_0$};

\end{tikzpicture}
\begin{center}
{\small{\sc Figure 1.} Distance of two points in the plane.}
\end{center}
\end{figure}

\item[(2)] As an example consider the city map of Manhattan and define as the distance of two places the kilometres you have to drive when you take the shortest route, cf.~Figure 2.\vspace{1pt}

\begin{figure}[h]
\hspace{-23pt}\begin{tikzpicture}
\fill [background,opacity=0.4] (0.1,0) rectangle (4,2.5);

\fill [black, opacity=0]  (-0.5,-0.25) rectangle (4,2.5);

\fill [grey-1] (0.1,2.5) rectangle (0.5,1.8);
\fill [grey-2] (0.7,2.5) rectangle (1.2,1.8);
\fill [grey-3, opacity=0.82] (1.4,2.5) rectangle (1.9,1.8);
\fill [grey-1] (2.1,2.5) rectangle (2.7,1.8);
\fill [grey-1] (2.9,2.5) rectangle (3.3,1.8);
\fill [grey-2] (3.5,2.5) rectangle (4.0,1.8);

\fill [grey-1] (0.1,1.6) rectangle (0.5,0.6);
\fill [grey-1] (0.7,1.6) rectangle (1.2,0.6);
\fill [grey-1] (1.4,1.6) rectangle (1.9,0.6);
\fill [grey-3, opacity=0.82] (2.1,1.6) rectangle (2.7,0.6);
\fill [grey-2] (2.9,1.6) rectangle (3.3,0.6);
\fill [grey-1] (3.5,1.6) rectangle (4.0,0.6);

\fill [grey-2] (0.1,0) rectangle (0.5,0.4);
\fill [grey-2] (0.7,0.4) rectangle (1.2,0);
\fill [grey-1] (1.4,0.4) rectangle (1.9,0);
\fill [grey-1] (2.1,0.4) rectangle (2.7,0);
\fill [grey-1] (2.9,0.4) rectangle (3.3,0);
\fill [grey-3, opacity=0.82] (3.5,0.4) rectangle (4.0,0);

\draw[green-1, fill=green-1] (0.6,1.7) to [out=115,in=-90] (0.47,2) to [out=90,in=180] (0.6,2.1);

\draw[green-1, fill=green-1] (0.6,1.7) to [out=65,in=-90] (0.73,2) to [out=90,in=360] (0.6,2.1);

\fill[white] (0.6,1.98) circle [radius=0.05];

\draw[ocker-1,line width=1.2pt, opacity=1]  (0.6,1.65) --  (0.6,0.5)  -- (3.34,0.5);

\draw[red-1, fill=red-1] (3.4,0.5) to [out=115,in=-90] (3.27,0.8) to [out=90,in=180] (3.4,0.9);

\draw[red-1, fill=red-1] (3.4,0.5) to [out=65,in=-90] (3.53,0.8) to [out=90,in=360] (3.4,0.9);

\fill[white] (3.4,0.78) circle [radius=0.05];
\end{tikzpicture}
\begin{center}
{\small{\sc Figure 2.} Distance of two points in a city map.}
\end{center}
\end{figure}

\item[(3)] A distance of two members can be defined as the minimal number of \textquotedblleft{}hops\textquotedblright{} which is needed to pass from member $A$ to member $B$ when it is only allowed to jump between members which are friends. In the network pictured in Figure 3 we have for instance $d(A,F)=4$.\vspace{1pt}

\begin{figure}[h]
\hspace{-37pt}\begin{tikzpicture}[scale=0.45]
  \SetVertexNormal[Shape      = circle,
                   FillColor  = background,
                   LineWidth  = 1pt,
                   LineColor = grey-4]
  \SetUpEdge[lw         = 1pt,
             color      = grey-4,
             labelcolor = background,
             labeltext  = red,
             labelstyle = {draw,text=blue}]
 \tikzstyle{EdgeStyle}=[]

\Vertex[x=0.5, y=1]{A}
\Vertex[x=3, y=-1]{B}
\Vertex[x=6, y=0.2]{C}
\Vertex[x=8, y=3]{D}
\Vertex[x=4.5, y=3]{E}
\Vertex[x=10, y=-0.1]{F}
\Vertex[x=1.8, y=3.4]{G}

\Edges(A,G,E,C) \Edges(B,E) 
 \SetUpEdge[lw         = 2pt,
             color      = ocker-1,
             labelcolor = background,
             labeltext  = red,
             labelstyle = {draw,text=blue}]
\Edges(A,B,C,D,F) 

\fill [black, opacity=0]  (-3,-2) rectangle (1,1);

\end{tikzpicture}
\begin{center}
{\small{\sc Figure 3.} Distance in a network.}
\end{center}
\end{figure}

\item[(4)] The distance of $f$ and $g$ can, e.g., be defined as the maximum of the modulus of their difference over their domain, i.e.,
$$
d(f,g)=\max_{0\leqslant{}x\leqslant{}1}|f(x)-g(x)|.
$$
In the example given, $f(x)=\sqrt{x}$, $g(x)=x^{3}$, $d(f,g)=5/(6\sqrt[5]{6})$ can be explicitly computed by methods usually taught at school, cf.~Figure 4 for a visualization of the distance.\vspace{3pt}

\begin{figure}[h]
\hspace{-27pt}\begin{tikzpicture}[xscale=5,yscale=2.5]
    \draw [<->,thick] (0,1) node (yaxis) [above] {}
        |- (1.1,0) node (xaxis) [right] {};

  \draw[scale=1,domain=0:1,smooth,variable=\x,black,line width=0.8pt] plot ({\x},{\x*\x*\x});
  \draw[scale=1,domain=0:1,smooth,variable=\y,black,line width=0.8pt]  plot ({\y*\y},{\y});

\coordinate (A) at (0.4883593419305869,0.116471186461929874205496183795236640584465386094976066048669);
\coordinate (X) at (0,0.116471186461929874205496183795236640584465386094976066048669);
\coordinate (B) at (0.4883593419305869,0.698827118771579245232977102771419843506792316569856396292015);
\coordinate (D) at (0,0.698827118771579245232977102771419843506792316569856396292015);

\coordinate (C) at (1,1);
\coordinate (S) at (0,1);

\draw[black, line width=1.2pt] (A)--(B);

\draw[dashed] (X)-- (A);
\draw[dashed] (D)-- (B);

\draw [decorate,decoration={brace,amplitude=4pt,mirror,raise=6.2pt},yshift=0pt] (D) -- (X) node [black,midway,xshift=-1.1cm,yshift=-0cm] {~~~\footnotesize$d(f,g)$};

\node[text width=1.5cm] at (0.7,0.99)  {\footnotesize$f(x)=\sqrt{x}$};
\node[text width=1.5cm] at (1.06,0.6)  {\footnotesize$g(x)=x^3$};

\draw (xaxis -| C) node[below] {$1$};
\draw (xaxis -| S) node[below] {$0$};

\end{tikzpicture}
\begin{center}
{\small{\sc Figure 4.} Distance of two functions.}
\end{center}
\end{figure}
\end{compactitem}

After presenting the above we turned to the next task for which we prepared an exercise sheet.

\begin{ta}\label{Exercise-2}Review the examples from Exercise \ref{Exercise-1} by working on the following tasks. \vspace{3pt}
\begin{compactitem}
\item[(a)] Decide whether the distance function $d=d(x,y)$ of the examples (1)--(4) satisfies the following properties, respectively.\vspace{3pt}
\begin{compactitem}
\item[1.]Distances are always real numbers, i.e., $d(x,y)\in\mathbb{R}$.\vspace{3pt}
\item[2.]Distances are always non-negative, i.e., $d(x,y)\geqslant0$.\vspace{3pt}
\item[3.]The distance of a point to itself is zero, i.e., $d(x,x)=0$.\vspace{3pt}
\item[4.]If the distance of two points is zero, then the points are equal, i.e., $d(x,y)=0\Rightarrow x=y$.\vspace{3pt}
\item[5.]The distance is independent of the order of the points, i.e., $d(x,y)=d(y,x)$.\vspace{3pt}
\item[6.] A detour can only increase the distance, more precisely, $d(x,y)\leqslant d(x,z)+d(z,y)$.\vspace{5pt}
\end{compactitem}
\item[(b)] Find a city map that does not satisfy the symmetry condition 5 in the situation of Exercise \ref{Exercise-1}(2).\vspace{3pt}
\item[(c)] Consider the network

\begin{center}
\hspace{-50pt}\begin{tikzpicture}[scale=0.45]
  \SetVertexNormal[Shape      = circle,
					LineColor = grey-4,
                   FillColor  = background,
                   LineWidth  = 1pt]
  \SetUpEdge[lw         = 1pt,
             color      = grey-4,
             labelcolor = white,
             labeltext  = red,
             labelstyle = {draw,text=blue}]
 \tikzstyle{EdgeStyle}=[]
 \Vertex[x=0, y=0]{$x_1$}
\Vertex[x=3, y=2]{$x_2$}
 \Vertex[x=6, y=0]{$x_3$}
\Edges($x_1$,$x_2$) 
\end{tikzpicture}
\end{center}

\medskip

\noindent{}and define $d(x_i,x_j)$ as in Exercise \ref{Exercise-1}(3) if $1\leqslant i,j\leqslant2$. Define the remaining distances such that condition 6 in (a) is not satisfied but all other conditions are valid.
\end{compactitem}
\end{ta}

\medskip

After a short period of work three of the students were asked to present their solutions. In case of (a) this was done orally, for (b) and (c) the participants used a flip chart to draw the corresponding picture (a city map with a one-way road) and to write down the definition (e.g., $d(x_1,x_3)=d(x_3,x_1)=3$ and $d(x_2,x_3)=d(x_3,x_2)=1$), respectively.

\medskip

Part (a) above was a preparation for the formal definition of the metric as we added mathematical expressions to the properties' formulation in prose.

\smallskip

\begin{dfn}\label{DFN} Let $M$ be a set. A function $d\colon M\times M\rightarrow\mathbb{R}$, $(x,y)\mapsto d(x,y)$ is said to be a \textit{metric} if the following conditions hold for all $x$, $y$ and $z$ in $M$.\vspace{5pt}
\\\begin{tabular}[h]{llll}\vspace{4pt}
(M1)& $d(x,y)\geqslant0$, and $d(x,y)=0\,\Leftrightarrow\,x=y$, & &(positive definiteness)\\\vspace{4pt}
(M2)& $d(x,y)=d(y,x)$, & &(symmetry)\\
(M3)& $d(x,y)\leqslant d(x,z) + d(z,y)$. & &(triangle inequality)\\
\end{tabular}
\end{dfn}

\medskip

In order to illustrate how to check the conditions in a concrete situation we considered first the absolute value metric and proved in detail that it is indeed a metric.

\medskip

\begin{thm}\label{THM} The function $d\colon\mathbb{R}\times\mathbb{R}\rightarrow\mathbb{R}$, $d(x,y)=|x-y|$ defines a metric on $\mathbb{R}$.
\end{thm}
\begin{proof} We have to check that the conditions (M1)--(M3) hold. Let $x$, $y$ and $z$ be real numbers.

\medskip

The estimate $d(x,y)\geqslant0$ is valid by definition. Moreover, $d(x,x)=|x-x|=0$ holds and $d(x,y)=0$ implies $|x-y|=0$ and thus $x=y$. This proves (M1).

\medskip

For (M2) it is enough to compute
$$
d(x,y)=|x-y|=|y-x|=d(y,x).
$$

Finally, we have
\begin{eqnarray*}
d(x,y) &=&|x-y|=|x-z+z-y|\\
       &\leqslant&|x-z|+|z-y|= d(x,z)+d(z,y).
\end{eqnarray*}
Here, we used $|a+b|\leqslant|a|+|b|$ with $a=x-z$ and $b=z-y$. The latter is equivalent to $(a+b)^2\leqslant(|a|+|b|)^2$ which in turn is equivalent to $a^2+2ab+b^2\leqslant a^2+2|a||b|+b^2$ which is obvious after cancellation. This completes the proof of (M3).
\end{proof}

The next step was the task for the participants to check the properties (M1)--(M3) in another example by themselves. Note that the following so-called Manhattan metric corresponds to the picture in Figure 2.

\smallskip

\begin{ta}\label{Exercise-3} Prove that $d\colon\mathbb{R}^2\times\mathbb{R}^2\rightarrow\mathbb{R}$, $d(x,y)=|x_1-y_1|+|x_2-y_2|$ is a metric.
\end{ta}

\smallskip

We leave the solution to the reader as we explained the major techniques in Theorem \ref{THM}. The par\-ti\-ci\-pants can re-use the proof given above. In order to show that the concept of a metric space is not just a purely theoretical object we gave the following application.

\begin{app}\label{APP} Metric spaces are \emph{the} setting in which convergence of sequences can be defined: Let $(x_n)_{n\in\mathbb{N}}$ be a sequence and $x$ a point in $M$. We say that $(x_n)_{n\in\mathbb{N}}$ converges to $x$, and write $\lim_{n\rightarrow\infty}x_n=x$, if for every number $\epsilon>0$ there exists an integer $n_0$ such that for all integers $n\geqslant n_0$ the inequality $d(x_n,x)<\epsilon$ holds.
\end{app}

The participants had encountered about limits of functions at school and were used to the notation $\lim_{x\rightarrow\infty}f(x)$. Our definition above is the special case where the domain of $f$ is the set of positive integers, i.e., $f\colon \mathbb{N}\rightarrow\mathbb{R}$, $f(n)=x_n$. This is, however, usually not discussed in German schools. In addition, limits of functions are not explained in a rigorous way. Nevertheless, we observed that our students had some intuitive idea about the concept of a limit.  We used Application \ref{APP} in order to point out that the notion of a metric space now allows for a formal definition. This is a clear advantage compared with an only intuitive idea. Moreover, it shows that convergence is not an intrinsic feature of the real numbers, but depends on the \emph{selection of a metric}.

\begin{ex}We consider the following two examples of different metrics on the real line.\vspace{3pt}
\begin{compactitem}
\item[(1)] Let $M=\mathbb{R}$, $d$ be the absolute value metric, $x_n=\frac{1}{n}$ and $x=0$. Let $\epsilon>0$ be arbitrary. We select $n_0>\frac{1}{\epsilon}$. Let $n\geqslant n_0$ be given. Then we have $d(x_n,x)=|\frac{1}{n}-0|=\frac{1}{n}\leqslant\frac{1}{n_0}<\epsilon$. Therefore, we proved $\lim_{n\rightarrow\infty}\frac{1}{n}=0$.\vspace{5pt}
\item[(2)] Let $M=\mathbb{R}$ and $q$ be the so-called trivial metric, i.e.,
$$
q\colon\mathbb{R}\times\mathbb{R}\rightarrow\mathbb{R},\;q(x,y)=\begin{cases} \;0 & \text{ if } x=y,\\ \;1 & \text{ otherwise.}\end{cases}
$$
For $(x_n)_{n\in\mathbb{N}}$ and $x$ as in (1) we then have $q(x_n,x)=q(\frac{1}{n},0)=1$ for any integer $n$. Thus, $\lim_{n\rightarrow\infty}\frac{1}{n}=0$ is not true with respect to the metric $q$. Convergence thus depends on the metric.
\end{compactitem}
\end{ex}

\smallskip

Finally, we turned back to the last example from Exercise \ref{Exercise-1} in order to complement the previous material by an example that really goes beyond the scope of school mathematics.

\medskip

\begin{ex} Let $M=C[0,1]$ be the set of all continuous functions from $[0,1]$ to $\mathbb{R}$. The metric $d(f,g)=\max_{0\leqslant x\leqslant1}|f(x)-g(x)|$ allows to study if a sequence of functions is convergent. The sequence $f_n(x)=\sin(\frac{x}{n})$ for instance converges to the zero function but for the sequence $g_n(x)=x^n$ there is no continuous function $g$ such that $\lim_{n\rightarrow\infty}g_n=g$ holds with respect to the metric $d$.
\end{ex}

\vspace{-7pt}

\begin{figure}[h]
\begin{tikzpicture}[xscale=5,yscale=2]
    \draw [<->,thick] (0,1.1) node (yaxis) [above] {}
        |- (1.1,0) node (xaxis) [right] {};

 \draw[scale=1,domain=0:1,smooth,variable=\x,black,line width=0.8pt]   plot (\x,{sin(\x*(1/2) r)});   

\draw[scale=1,domain=0:1,smooth,variable=\x,black,line width=0.8pt]   plot (\x,{sin(\x r)});

 \draw[scale=1,domain=0:1,smooth,variable=\x,black,line width=0.8pt]   plot (\x,{sin(\x*(1/4) r)}); 
 
  \draw[scale=1,domain=0:1,smooth,variable=\x,black,line width=0.8pt]   plot (\x,{sin(\x*(1/10) r)});

\coordinate (C) at (1,1);
\coordinate (S) at (0,1);

\node[text width=2cm] at (0.67,0.9)  {\footnotesize$f_n(x)=\sin(\frac{x}{n})$};

\draw (xaxis -| C) node[below] {$1$};
\draw (xaxis -| S) node[below] {$0$};

\end{tikzpicture}
\begin{tikzpicture}[xscale=5,yscale=2]
    \draw [<->,thick] (0,1.1) node (yaxis) [above] {}
        |- (1.1,0) node (xaxis) [right] {};
\draw[scale=1,domain=0:1,smooth,variable=\x,black,line width=0.8pt] plot ({\x},{\x});

\draw[scale=1,domain=0:1,smooth,variable=\x,black,line width=0.8pt] plot ({\x},{\x*\x});
\draw[scale=1,domain=0:1,smooth,variable=\x,black,line width=0.8pt] plot ({\x},{\x*\x*\x*\x*\x});
\draw[scale=1,domain=0:1,smooth,variable=\x,black,line width=0.8pt] plot ({\x},{\x*\x*\x*\x*\x*\x*\x*\x*\x*\x*\x*\x});

\draw[scale=1,domain=0:1,smooth,variable=\x,black,line width=0.8pt] plot ({\x},{\x*\x*\x*\x*\x*\x*\x*\x*\x*\x*\x*\x*\x*\x*\x*\x*\x*\x*\x*\x*\x*\x*\x*\x});
\coordinate (C) at (1,1);
\coordinate (S) at (0,1);
\node[text width=2cm] at (0.67,0.9)  {\footnotesize$g_n(x)=x^n$};
\draw (xaxis -| C) node[below] {$1$};
\draw (xaxis -| S) node[below] {$0$};

\fill [black, opacity=0]  (-0.15,-0.25) rectangle (1,1);

\end{tikzpicture}
\begin{center}
{\small{\sc Figure 5.} A convergent sequence of functions (left) and a non-convergent sequence (right).}
\end{center}
\end{figure}

For the last example we did not give a proof, but indicated pictures of the $f_n$ and $g_n$, see Figure 5, so as to create a visual understanding for the statement.

\section{Reflection}\label{SEC-3}

In Section 1, the goal of our workshop was described as painting a realistic picture of mathematical studies at university. Our workshop met this goal in two ways: The first way is concerned with a theoretically based analysis of the content of the workshop. It inherited lecture parts which not only touched university level but also provided a realistic insight into basic skills and methods used in mathematics: Definition \ref{DFN}, Theorem \ref{THM} and the notion of convergence given in Application \ref{APP} appear verbatim in many first year courses on analysis. The participants solved exercises, in particular Exercise \ref{Exercise-3} is compatible with homework for first year university students in the German system. Finally, the participants presented the solutions to their classmates as it is usually required by university students during exercise sessions. Due to the concept of working in small groups with one teaching assistant assigned to each group it was possible to provide the participants not only with an experience what mathematical studies are, but also with an experience what mathematics is about.

\medskip

The second way in which the workshop met our goal stated in Section 1 is related to students' written feedback. After the workshop, they were asked to answer the following prompts:

\vspace{5pt}

\begin{compactitem}
\item[(1)] What did you like? What did you understand well?\vspace{3pt}
\item[(2)] What didn't you like? What didn't you understand well?\vspace{3pt}
\item[(3)] What additional comments do you have?\vspace{3pt}
\end{compactitem}

\smallskip

Two students mentioned that they liked that the workshop was close to university studies:\vspace{-5pt}
\begin{figure}[h]
\includegraphics[width=300pt]{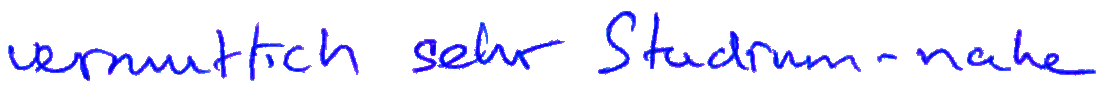}
\begin{center}
\textquotedblleft{}probably very similar to university studies\textquotedblright{}\vspace{5pt}
\end{center}
\includegraphics[width=160pt]{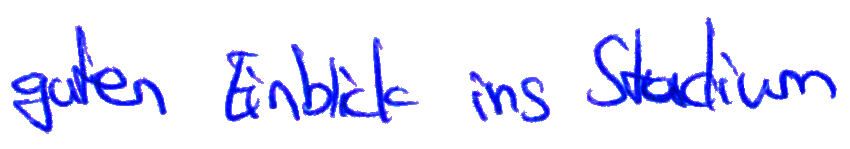}
\begin{center}\vspace{-3pt}
\textquotedblleft{}good insight into university studies\textquotedblright{}
\end{center}
\end{figure}

Five students explicitly mentioned that they liked the active participation via exercises, e.g.:
\begin{figure}[h]
\includegraphics[width=380pt]{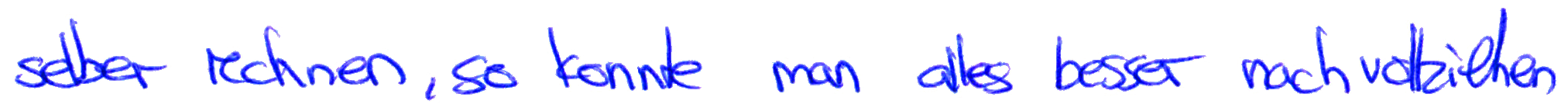}
\begin{center}
\textquotedblleft{}by computing ourselves we could understand everything much better\textquotedblright{}\vspace{5pt}
\end{center}
\includegraphics[width=350pt]{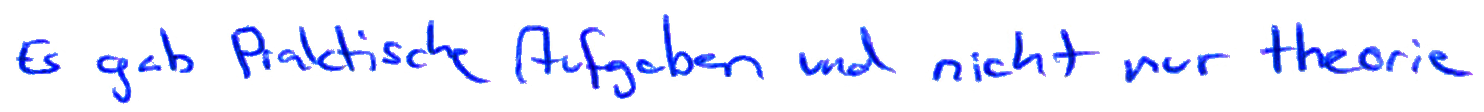}
\begin{center}\vspace{-3pt}
\textquotedblleft{}concrete exercises and not only theory\textquotedblright{}
\end{center}
\end{figure}

\pagebreak

Eight students appreciated working in small groups and explicitly mentioned it, e.g.:

\begin{figure}[h]
\includegraphics[width=269pt]{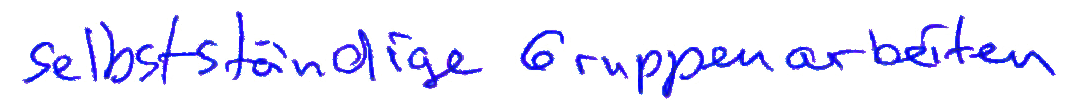}
\begin{center}
\textquotedblleft{}working independently in groups\textquotedblright{}
\end{center}
\end{figure}

With regards to criticism, two participants indicated that the topic was difficult, eight mentioned that the lecture parts were too fast and that there was not enough time for the exercises though one student, at the same time, admitted in a rather rational way that this is \textquotedblleft{}probably due to university studies\textquotedblright{}. However, the latter impressions are a characteristic aspect of studies in mathematics and it is therefore a good idea honestly to inform prospective students  about this issue before they finally decide which major to take.

\section{Conclusion}\label{SEC-4}

As we explained in Section \ref{SEC-3} the workshop was very successful. It met our goals and it received an encouraging feedback from the participants. In particular, assigning exercises to the participants seemed to be a key factor for the success of the workshop. This corresponds, e.g., to the results of Halverscheid and Sibbertsen \cite[p.~8]{HS} who found in a similar format that \textquotedblleft{}exercises are most important for the students\textquotedblright{}. We can thus recommend the workshop format as well as our concrete topic to anybody interested in organizing a similar open day in mathematics.

\medskip

\section*{Acknowledgements}\vspace{-2pt}

{

\small

The authors would like to thank M.\,Frantzen, F.\,Hasenkamp, K.\,Kerkmann, E.\,Mol\'{n}ar and N.\,Vaz for their help during the workshop. In addition they would like to thank the participants for their active participation, which made the workshop very enjoyable also for the teachers. Finally, the authors would like to thank the reviewers for their very valuable comments.

}

\medskip

\normalsize

\bibliographystyle{amsplain}

\providecommand{\href}[2]{#2}


\begin{thebibliography}{1}

\bibitem{AE1}
H.~Amann and J.~Escher, \emph{Analysis. {I}}, Birkh\"auser Verlag, Basel, 2005.

\bibitem{Browder}
A.~Browder, \emph{Mathematical analysis}, Undergraduate Texts in Mathematics,
  Springer-Verlag, New York, 1996.

\bibitem{D}
J.~Dugundji, \emph{Topology}, Allyn and Bacon, Inc., Boston,
  Mass.-London-Sydney, 1966.
  
  \bibitem{HS}
S.~Halverscheid and P.~Sibbertsen, \emph{An introduction to {M}arkov chains for
  interested high school students}, Technical report: Sonderforschungsbereich
  Komplexit{\"a}tsreduktion in Multivariaten Datenstrukturen, Univ., SFB 475,
  2003.

\bibitem{Krause}
E.~F.~Krause, \emph{Taxicab geometry: an adventure in non-{E}uclidean geometry}, Addison-Wesley, 1975.
  
\bibitem{Rudin}
W.~Rudin, \emph{Principles of mathematical analysis}, third ed., McGraw-Hill
  Book Co., New York-Auckland-D\"usseldorf, 1976, International Series in Pure
  and Applied Mathematics.

\end{thebibliography}
\end{document}